\documentclass[12pt]{article}

\usepackage{latexsym}
\usepackage{amsmath}
\usepackage{color}
\usepackage{adjustbox}
\usepackage{natbib}
\setcitestyle{citesep={;}, aysep={,}}

\usepackage{tikz}
\usepackage{float}
\usepackage{subfig}
\usepackage{caption} 

\usepackage{floatpag}
\floatpagestyle{empty}

\numberwithin{equation}{section}

\newtheorem{theorem}{Theorem}[section]
\newtheorem{corollary}{Corollary}[section]
\newtheorem{definition}{Definition}[section]

\newtheorem{remark}{Remark}[section]

\numberwithin{equation}{section}
\newtheorem{proposition}{Proposition}[section]

\newenvironment{proof}[1][Proof]{\textbf{#1.} }{\
\qquad\qquad\rule{0.5em}{0.5em }}

\newcommand{\RNum}[1]{\uppercase\expandafter{\romannumeral #1\relax}}

\DeclareMathOperator*{\argmax}{arg\,max}

\newtheorem{lemma}{Lemma}[section]

\newcommand{\rom}[1]{\uppercase\expandafter{\romannumeral #1\relax}}

\title{On Success runs of a fixed length defined on a $q$-sequence of binary trials
}
\author{Jungtaek Oh, Dae-Gyu Jang}

\begin{document}

\maketitle
\begin{abstract}
We study the exact distributions of runs of a fixed length in variation which considers binary trials for which the probability of ones is geometrically varying. The random variable $E_{n,k}$ denote the number of success runs of a fixed length $k$, $1\leq k \leq n$.
 Theorem \ref{exactpmftype4q-bino} gives an closed expression for the probability mass function (PMF) of the Type$\rom{4}$ $q$-binomial distribution of order $k$.  Theorem \ref{recurpmftype4q-bino} and Corollary \ref{recurpmftype4q-bino2} gives an recursive expression for the probability mass function (PMF) of the Type$\rom{4}$ $q$-binomial distribution of order $k$. The probability generating function and moments of random variable $E_{n,k}$ are obtained as a  recursive expression. We address the parameter estimation in the distribution of $E_{n,k}$ by numerical techniques. In the present work, we consider a sequence of independent binary zero and one trials with not necessarily identical distribution with the probability of ones varying according to a geometric rule. Exact and recursive formulae for the distribution obtained by means of enumerative combinatorics.

\end{abstract}

\tableofcontents

\section{Introduction}
\citet{charalambides2010a} studied discrete $q$-distributions on Bernoulli trials with a geometrically varying success probability. Let us consider a sequence $X_{1}$,...,$X_{n}$ of zero(failure)-one(success) Bernoulli trials, such that the trials of the subsequence after the $(i-1)$st zero until the $i$th zero are independent with equal failure probability. The $i$'s geometric sequences of trials is the subsequence after the $(i-1)$'st zero and until
the $i$'th zero, for $i>0$ and the subsequence after the $(j-1)$'st zero and until
the $j$'th zero, for $j>0$ are independent for all $i\neq j$ (i.e. $i$'th and $j$'th geometric sequences are independent) with probability of zeros at the $i$th geometric sequence of trials
\noindent
\begin{equation}\label{failureprobofq}
\begin{split}
q_{i}=1-\theta q^{i-1},\quad i=1,2,..., \quad 0\leq\theta\leq1,\quad 0\leq q<1.
\end{split}
\end{equation}
We note that probability of failures in the independent geometric sequences of trials is geometrically increasing with rate $q$.
Let $S_{j}^{(0)}=\Sigma_{m=1}^{j}(1-X_{m})$ denote the number of zeros in the first $j$ trials. Because the probability of zero's at the $i$th geometric sequence of trials is in fact the conditional probability of occurrence of a zero at any trial $j$ given the occurrence of $i-1$ zeros in the previous trials. We can rewrite as follows.
\noindent
\begin{equation}\label{failureprobofq2}
\begin{split}
q_{j,i}=p\Big(X_{j}=0\ \bigm|\ S_{j-1}^{(0)}=i-1\Big)=1-\theta q^{i-1},\quad i=1,2,...,j, \quad j=1,2,....
\end{split}
\end{equation}
\noindent
We note that \eqref{failureprobofq} is exactly the conditional probability in \eqref{failureprobofq2}. To make more clear and transparent the preceding, we consider an example $n = 18$, the binary sequence 111011110111100110, each subsequence has own success and failure probabilities according to a geometric rule.

\begin{figure}[t]
\noindent
  \begin{center}
\includegraphics[scale=0.65]{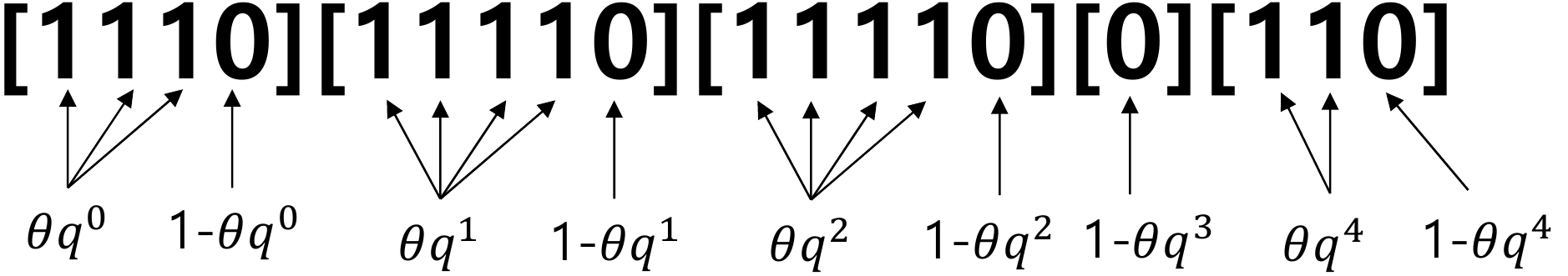}
\end{center}
\end{figure}

This stochastic model \eqref{failureprobofq} or \eqref{failureprobofq2} has interesting applications, studied as a reliability growth model by \cite{Dubman&Sherman1969}, and applies to a $q$-boson theory in physics by \cite{Jing&Fan1994} and \cite{Jing1994}. More specifically, $q$-binomial distribution introduced as a $q$-deformed binomial distribution, in order to set up a $q$-binomial state. This stochastic model \eqref{failureprobofq} also applies to start-up demonstration tests, as a sequential-intervention model which is proposed by \cite{BBV1995}.\\

The stochastic model \eqref{failureprobofq} is $q$-analogue of the classical binomial distribution with geometrically varying probability of zeros, which is a stochastic model of an independent and identically distributed (IID) trials with failure probability is
\noindent
\begin{equation}\label{bernoullifailureprob}
\pi_{j}=P\left(X_{j}=0\right)=1-\theta,\ j=1,2,\ldots,\ 0<\theta<1.
\end{equation}
\noindent
As $q$ tends toward $1$, the stochastic model \eqref{failureprobofq} reduces  to IID(Bernoulli) model \eqref{bernoullifailureprob}, since $q_{i}\rightarrow\pi_{i}$, $i=1,2,\ldots$ or $q_{j,i}\rightarrow1-\theta$, $i=1,2,\ldots,j$, $j=1,2,\ldots.$\\

The Discrete $q$-distributions based on the stochastic model of the sequence of independent Bernoulli trials have been investigated by numerous researchers, for a lucid review and comprehensive list of publications on this area the interested reader may consult the monographs by \cite{charalambides2010a,Charalambides2010b,Charalambides2016}.\\

From a Mathematical and Statistical point of view, \cite{Charalambides2016} mentioned the preface of his book \textit{"It should be noticed that a stochastic model of a sequence of independent Bernoulli trials, in which the probability of success at a trial is assumed to vary with the number of trials and/or the number of successes, is advantageous in the sense that it permits incorporating the experience gained from previous trials and/or successes. If the probability of success at a trial is a very general function of the number of trials and/or the number successes, very little can be inferred from it about the distributions of the various random variables that may be defined on this model. The assumption that the probability of success (or failure) at a trial varies geometrically, with rate (proportion) $q$, leads to the introduction of discrete $q$-distributions"}.\\

The distribution theory of runs and patterns has been incredibly developed in the last few decades through a slew of the research literature because of their theoretical interest and applications in a wide variety of research areas such as hypothesis testing, system reliability, quality control, physics, psychology, radar astronomy, molecular biology, computer science, insurance, and finance. During the past few decades up to recently, the meaningful progress on runs and pattern statistics has been wonderfully surveyed in \citet{balakrishnan2003runs} as well as in \citet{fu2003distribution} and references therein. Furthermore, there are some more recent contributions on the topic such as \citet{arapis2018}, \citet{eryilmaz2018}, \citet{kong2019}, \citet{makri2019}, and \citet{aki2019}.\\
There are several ways of counting scheme. Each counting scheme depends on different conditions: whether or not the overlapping counting is permitted, and whether or not the counting starts from scratch when a certain kind or size of run has been so far enumerated. Feller (1968)\cite{feller1968introduction} proposed a classical counting method, once $k$ consecutive successes show up, the number of occurrences of $k$ consecutive successes is counted and the counting procedure starts anew, called \textit{non-overlapping} counting scheme which is referred to as Type $\rom{1}$ distributions of order $k$. A second scheme can be initiated by counting a success runs of length greater than or equal to $k$ preceded and followed by a failure or by the beginning or by the end of the sequence (see. e.g. Mood, 1940 or Gibbons, 1971 or Goldstein, 1990) and is usually called \textit{at least} counting scheme which is referred to as Type $\rom{2}$ distributions of order $k$. Ling (1988) suggested the \textit{overlapping} counting scheme, an uninterrupted sequence of $m\geq k$ successes preceded and followed by a failure or by the beginning or by the end of the sequence. It accounts for $m-k+1$ success runs of length of $k$ which is referred to as Type $\rom{3}$ distributions of order $k$. Mood(1940) suggested \textit{exact} counting scheme, asuccess run of length exactly $k$ preceded and succeeded by failure or by nothing which is referred to as Type $\rom{4}$ distributions of order $k$.

According to the three aforementioned counting schemes,the random variables of the number of runs of length $k$ counted in $n$ outcomes, have three different distributions which are denoted as $N_{n,k}$, $G_{n,k}$, $M_{n,k}$ and $E_{n,k}$. Moreover, if the underline sequence is an independent and identically distributed (i.i.d.) sequence of random variables, $X_{1},X_{2},\ldots ,X_{n}$, then distributions of $N_{n,k}$, $G_{n,k}$, $M_{n,k}$ and $E_{n,k}$  will be referred to as Type $\rom{1}$, $\rom{2}$, $\rom{3}$ and $\rom{4}$  binomial distributions of order $k$.

To make more clear the distinction between the aforementioned counting methods we mention by way of example that for $n=12$, the binary sequence $011111000111$ contains $N_{12,2}=3$, $G_{12,2}=2$, $M_{12,2}=6$,
 $T_{2,2}^{(\rom{1})}=5$, $T_{2,2}^{(\rom{2})}=11$, and $T_{2,2}^{(\rom{3})}=4$.

\section{Preliminary and Notation}
We first recall some definitions, notation and known results in which will be used in this paper. Throughout the paper, we suppose that $0<q<1$. First, we introduce the following notation.
\begin{itemize}
\item $L_n^{(1)}:$ the length of the longest run of successes in $X_{1}, X_{2}, \ldots, X_{n}$;
\item $L_n^{(0)}:$ the length of the longest run of failures in $X_{1}, X_{2}, \ldots, X_{n}$;
\item $S_{n}:$ the total number of successes in $X_{1}, X_{2}, \ldots, X_{n}$;
\item $F_{n}:$ the total number of failures in $X_{1}, X_{2}, \ldots, X_{n}$.
\end{itemize}
\noindent
Next, let us introduce some basic $q$-sequences and functions and their properties, which are useful in the sequel. The $q$-shifted factorials are defined as
\begin{equation}\label{}
(a;q)_{0}=1,\quad (a;q)_{n} =\prod_{k=0}^{n-1}(1-aq^{k}),\quad (a;q)_{\infty} =\prod_{k=0}^{\infty}(1-aq^{k}).
\end{equation}
\noindent
Let $m$, $n$ and $i$ be positive integer and $z$ and $q$ be real numbers, with $q\neq 1$. The number $[z]_{q} = (1-q^{z})/(1-q)$ is called $q$-number and in particular $[z]_{q}$ is called $q$-integer. The $m$ th order factorial of the $q$-number $[z]_{q}$, which is defined by
\begin{equation}\label{}
\begin{split}
[z]_{m,q}&=\prod_{i=1}^{m}[z-i+1]_{q}= [z]_{q} [z-1]_{q}\cdots[z-m+1]_{q}\\
&=\frac{(1-q^{z})(1-q^{z-1})\cdots(1-q^{z-m+1})}{(1-q)^{m}},\ z=1,2,\ldots,\ m=0,1,\ldots,z.
\end{split}
\end{equation}
\noindent
is called $q$-factorial of $z$ of order $m$. In particular, $[m]_{q}! = [1]_{q}[2]_{q}...[m]_{q}$ is called $q$-factorial of $m$. The $q$-binomial coefficient (or Gaussian polynomial) is defined by
\noindent
\begin{equation}\label{}
\begin{split}
\begin{bmatrix}
n\\
m
\end{bmatrix}_{q}
&=\frac{[n]_{m,q}}{[m]_{q}!}=\frac{[n]_{q}!}{[m]_{q}![n-m]_{q}!}=\frac{(1-q^{n})(1-q^{n-1})\cdots(1-q^{n-m+1})}{(1-q^{m})(1-q^{m-1})\cdots(1-q)}\\
&=\frac{(q;q)_{n}}{(q;q)_{m}(q;q)_{n-m}},\ m=1,2,\ldots,
\end{split}
\end{equation}
\noindent
The $q$-binomial ($q$-Newton's binomial) formula is expressed as
\noindent
\begin{equation}\label{}
\prod_{i=1}^{n}(1+zq^{i-1})=\sum_{k=0}^{n}q^{k(k-1)/2}\begin{bmatrix}
n\\
k
\end{bmatrix}_{q}
z^{k},\ -\infty<z<\infty,\ n=1,2,\ldots.
\end{equation}
\noindent
For $q\rightarrow 1$ the $q$-analogs tend to their classical counterparts, that is
\noindent
\begin{equation*}\label{eq: 1.1}
\begin{split}
\lim\limits_{q\rightarrow 1}
\begin{bmatrix}
n\\
r
\end{bmatrix}_{q}
={n \choose r}
\end{split}
\end{equation*}
\noindent
Let us consider again a sequence of independent geometric sequences of trials with probability of failure at the $i$th geometric sequence of trials given by \eqref{failureprobofq} or \eqref{failureprobofq2}. We are interesting now is focused on the study of the number of successes in a given number of trials in this stochastic model.
\begin{definition}
Let $Z_{n}$ be the number of successes in a sequence of $n$ independent Bernoulli trials, with probability of success at the $i$th geometric sequence of trials given by \eqref{failureprobofq} or \eqref{failureprobofq2}. The distribution of the random variable $Z_{n}$ is called $q$-binomial distribution, with parameters $n$, $\theta$, and $q$.
\end{definition}
\noindent
Let we introduce a $q$-analogue of the binomial distribution with the probability function of the number $Z_{n}$ of successes in $n$ trials $X_{1}, \ldots, X_{n}$ is given by
\begin{equation}\label{q-binomialdist}
\begin{split}
P_{q,\theta}\{Z_{n}=r\}=
\begin{bmatrix}
n\\
r
\end{bmatrix}_{q}
\theta^{r}\prod_{i=1}^{n-r}(1-\theta q^{i-1}),
\end{split}
\end{equation}\\
for $r = 0, 1, . . . , n$, $0 < q < 1$. The distribution  is called a $q$-binomial distribution. For $q \rightarrow 1$, because
\begin{equation*}\label{eq: 1.1}
\begin{split}
\lim\limits_{q\rightarrow 1}
\begin{bmatrix}
n\\
r
\end{bmatrix}_{q}
={n \choose r}
\end{split}
\end{equation*}
so that the q-binomial distribution converges to the usual binomial distribution as $q \rightarrow 1$, as follows
\begin{equation}
P_{\theta}\left(Z_{n}=r\right)={n\choose r}\theta^{r}(1-\theta)^{n-r},\ r=0,1,\ldots,n,
\end{equation}
with parameters $n$ and $\theta$.
The $q$-binomial distribution studied by \citet{charalambides2010a,Charalambides2016}, which is connected with $q$-Berstein polynomial. \citet{Jing1994} introduced probability function \eqref{q-binomialdist} as a $q$-deformed binomial distribution, also derived recurrence relation of its probability distribution.
\noindent
In the sequel, $P_{q ,\theta}(.)$ and $P_{\theta}(.)$ denote probabilities related with the stochastic model \eqref{failureprobofq} and \eqref{bernoullifailureprob}, respectively.

\section{Type $\rom{4}$ $q$-binomial distribution of order $k$} 
\vspace{0.13in}
We now make some useful Definition and Lemma for the proofs of Theorem in the sequel.
\begin{definition}
 For $0<q\leq1$, define the polynomial \begin{equation*}\label{b_qequation} \begin{split}
 B_{q}(r,s,t,k)=&\sum_{y_{1},y_{2},\ldots,y_{r}}   q^{y_{2}+2y_{3}+\cdots+(r-1)y_{r}}
\end{split}. \end{equation*} where the summation is taken place over all  $y_{1},y_{2},\ldots,y_{r}$ satisfying the conditions \begin{equation}
\begin{split} y_{i}\geq 0,\ i=1,2,\ldots,r,\ \sum_{i=1}^{r}y_{i}=s,\ \delta_{k,y_{1}}+\cdots+\delta_{k,y_{r}}=t,\ \delta_{i,j}=\left\{
               \begin{array}{ll}
                 1, & \text{if}\ i=j \\
                 0, & \text{if}\ i\neq j.
               \end{array}
             \right.
\end{split} \end{equation}
\end{definition}
\noindent
The following gives a recurrence relation useful for the computation of $B_{q}^{k}(r,s,t)$.\\
\noindent
\begin{lemma}
For $0<q\leq1$, $B_{q}(r,s,t,k)$ obeys the following recurrence relation, \begin{equation*}\label{eq: 1.1} \begin{split} B_{q}(r,s,t,k)= \left\{
  \begin{array}{ll}
    1, & \text{for}\ r=1,\ s=k,\ t=1 \\
    & \text{or}\ r=1,\ 0\leq s< k,\ t=0\\
        & \text{or}\ r=1,\ s> k,\ t=0\\
\sum_{j=0}^{k-1}q^{j(r-1)}B_{q}(r-1,s-j,t,k)\\ +q^{k(r-1)}B_{q}(r-1,s-k,t-1,k)\\ +\sum_{j=k+1}^{s}q^{j(r-1)}B_{q}(r-1,s-j,t,k), & \text{for}\
r\geq 2,\ s\geq tk,\ t\leq r \\
    0, & \text{otherwise.}\\
  \end{array}
\right.
\end{split}
\end{equation*}
\end{lemma}

\begin{proof}
For $r\geq 2$, $s\geq tk$, $t\leq r$, we observe that since $y_{r}$ can take the values $0,1,\ldots,s$, then $B_{q}(r,s,t,k)$ can be written as
\begin{equation}
\begin{split}
 B_{q}(r,s,t,k)=&{\mathop{\sum...\sum}_{\substack{{y_{1}}+\cdots+{y_{r-1}}=s\\
\delta_{k,y_{1}}+\cdots+\delta_{k,y_{r-1}}=t}}}   q^{y_{2}+2y_{3}+\cdots+(r-2)y_{r-1}}\\
&+q^{r-1}{\mathop{\sum...\sum}_{\substack{{y_{1}}+\cdots+{y_{r-1}}=s-1\\ \delta_{k,y_{1}}+\cdots+\delta_{k,y_{r-1}}=t}}}
q^{y_{2}+2y_{3}+\cdots+(r-2)y_{r-1}}\\ &+\cdots+q^{k(r-1)}{\mathop{\sum...\sum}_{\substack{{y_{1}}+\cdots+{y_{r-1}}=s-k\\
\delta_{k,y_{1}}+\cdots+\delta_{k,y_{r-1}}=t-1}}}   q^{y_{2}+2y_{3}+\cdots+(r-2)y_{r-1}}\\
&+q^{(k+1)(r-1)}{\mathop{\sum...\sum}_{\substack{{y_{1}}+\cdots+{y_{r-1}}=s-k-1\\ \delta_{k,y_{1}}+\cdots+\delta_{k,y_{r-1}}=t}}}
q^{y_{2}+2y_{3}+\cdots+(r-2)y_{r-1}}\\ &+\cdots+q^{s(r-1)}{\mathop{\sum...\sum}_{\substack{y_{1}+\cdots+y_{r-1}=0 \\
\delta_{k,y_{1}}+\cdots+\delta_{k,y_{r-1}}=t}}}   q^{y_{2}+2y_{3}+\cdots+(r-2)y_{r-1}}\\
\end{split}
\end{equation}
\noindent
Using simple algebraic arguments to simplify,
\noindent
\begin{equation*}
\begin{split}
 B_{q}(r,s,t,k)=&\sum_{y_{r}=0}^{k-1}q^{y_{r}(r-1)}{\mathop{\sum...\sum}_{\substack{{y_{1}}+{y_{2}}+\cdots+{y_{r-1}}=s-y_{r}\\
\delta_{k,y_{1}}+\delta_{k,y_{2}}+\cdots+\delta_{k,y_{r-1}}=t}}}   q^{y_{2}+2y_{3}+\cdots+(r-2)y_{r-1}}\\
&+q^{k(r-1)}{\mathop{\sum...\sum}_{\substack{{y_{1}}+{y_{2}}+\cdots+{y_{r-1}}=s-k\\
\delta_{k,y_{1}}+\delta_{k,y_{2}}+\cdots+\delta_{k,y_{r-1}}=t-1}}}   q^{y_{2}+2y_{3}+\cdots+(r-2)y_{r-1}}\\
&+\sum_{y_{r}=k+1}^{s}q^{y_{r}(r-1)}{\mathop{\sum...\sum}_{\substack{{y_{1}}+{y_{2}}+\cdots+{y_{r-1}}=s-y_{r}\\
\delta_{k,y_{1}}+\delta_{k,y_{2}}+\cdots+\delta_{k,y_{r-1}}=t}}}   q^{y_{2}+2y_{3}+\cdots+(r-2)y_{r-1}}\\
=&\sum_{j=0}^{k-1}q^{j(r-1)}B_{q}(r-1,s-j,t,k)+q^{k(r-1)}B_{q}(r-1,s-k,t-1,k)\\ &+\sum_{j=k+1}^{s}q^{j(r-1)}B_{q}(r-1,s-j,t,k)
\end{split}
\end{equation*}
\noindent
The other cases are obvious. \end{proof}
\noindent
\begin{remark}
{\rm
We notice that for $q=1$, the quantity $B_{1}(r,s,t,k)$ represents the number of integer solutions $(y_{1},y_{2},\ldots,y_{r})$ satisfying the system of Eqs \eqref{b_qequation}.  or, altenatively, it is the number of allocations of $s$ balls into $r$ cells so that each of exactly $t$
of them receives exaactly equal to $k$ balls. This number is given by \begin{equation*}\label{eq: 1.1} \begin{split} B_{1}(r,s,t,k)={r \choose
t}A(s-tk,\ r-t,\ k) \end{split}, \end{equation*} where (see Makri et a;. 2007) $A(\alpha,r,k)=\sum_{j=0}^{[\alpha/k]}(-1)^{j}{r \choose
j}{\alpha-(k+1)j+r-1 \choose \alpha-jk}$.
}
\end{remark}

\begin{theorem}
\label{exactpmftype4q-bino}
For $0<q\leq1$, the probability mass function of success runs with length exactly equal to $k$ in $n$ trials is given by
\begin{equation*}\label{eq: 1.1} \begin{split} P(E_{n,k}=x)=\sum_{i=0}^{n-xk}\theta^{n-i}{\prod_{{j}=1}^i}(1-\theta q^{j-1})
B_{q}(i+1,n-i,x,k). \end{split} \end{equation*} \end{theorem}

\begin{proof} Let $S_{n}$ denote the total number of failures in the first $n$ trials. For $i=1,\ldots,n-xk$, a typical element of the event $\{E_{n,k}=x,\ S_{n}=i\}$
is an ordered sequence which consists of $n-i$ successes and $i$ failures such that the length of success run is non-negative integer and exactly $x$ of them are length is equal to $k$. The number of these sequences can be derived as follows. First we will distribute the $i$ failures. Since $i$ failures form $i+1$ cells. Next, we will distribute the $n-i$ successes into the $i+1$ distinguishable cells as follows.
\begin{equation*}\label{eq: 1.1} \begin{split}
{\underbrace{1\ldots1}_{y_{1}}}0{\underbrace{1\ldots1}_{y_{2}}}0\ldots0{\underbrace{1\ldots1}_{y_{i+1}}} \end{split} \end{equation*} with $i$ 0s
and $n-i$ 1s, where the length of the first 1-run is $y_{1}$, the length of the second 1-run is $y_{2}$,..., the length of the $(i+1)$-th 1-run is
$y_{i+1}$. Then the probability of the above sequence is given by \noindent \begin{equation*}\label{eq: 1.1} \begin{split} (\theta
q^{0})^{y_{1}}&(1-\theta q^{0})(\theta q^{1})^{y_{2}}(1-\theta q^{1})\cdots(\theta q^{i-1})^{y_{i}}(1-\theta q^{i-1})(\theta q^{i})^{y_{i+1}}.
\end{split} \end{equation*}
\noindent
and then using simple exponentiation algebra arguments to simplify, \noindent \begin{equation*}\label{eq: 1.1}
\begin{split} \theta^{n-i}{\prod_{{j}=1}^i}(1-\theta q^{j-1})q^{y_{2}+2y_{3}+\cdots+iy_{i+1}}. \end{split} \end{equation*}
\noindent
But $y_{j}$s are nonnegative integers such that ${y_{1}}+{y_{2}}+\cdots+{y_{i+1}}=n-i$ and exactly $x$ of the $y_{j}$s lengths are equal to $k$ so that
\noindent
\begin{equation*}\label{eq: 1.1} \begin{split} P&(E_{n,k}=x,\ S_{n}=i)\\ =&\theta^{n-i}{\prod_{{j}=1}^i}(1-\theta q^{j-1})
{\mathop{\sum...\sum}_{\substack{{y_{1}}+{y_{2}}+\cdots+{y_{i+1}}=n-i\\ \delta_{k,y_{1}}+\delta_{k,y_{2}}+\cdots+\delta_{k,y_{i+1}}=x}}}
q^{y_{2}+2y_{3}+\cdots+iy_{i+1}}. \end{split} \end{equation*}
Using lemma, we can rewrite as follows
\begin{equation*}
\label{eq: 1.1}
\begin{split}
P&(E_{n,k}=x,\ S_{n}=i)\\ =&\theta^{n-i}{\prod_{{j}=1}^i}(1-\theta q^{j-1})
B_{q}(i+1,n-i,x,k).
\end{split}
\end{equation*}
\noindent
Summing with respect to $i=1,\ldots,n-xk$ the result follows.
\end{proof}

\begin{remark}
{\rm
We notice that for $q=1$, the PMF $P_{q,\theta}(E_{n,k}=x)$ converges to the PMF $P_{\theta}(E_{n,k}=x)$ as follow
\begin{equation*}\label{eq: 1.1} \begin{split} P_{\theta}(E_{n,k}=x)=&\sum_{i=0}^{n-xk}\theta^{n-i}(1-\theta )^{i}  B_{1}(i+1,n-i,x,k)\\
=&\sum_{i=0}^{n-xk}\theta^{n-i}(1-\theta )^{i}  {i+1 \choose x}A(n-i-xk,i+1-x,k),\\
&\quad \quad\quad\quad\quad\quad\quad\quad\quad\quad\quad\quad\quad \text{for}\ x=1,\ldots,\left[\frac{n+1}{k+1}\right], \end{split}
\end{equation*}
 where (see Makri et a;. 2007) $A(\alpha,r,k)=\sum_{j=0}^{[\alpha/k]}(-1)^{j}{r \choose
j}{\alpha-(k+1)j+r-1 \choose \alpha-jk}$.
}
\end{remark}

We will observe how the probability of failure varies according on the model , after the first occurrence of the failure. We obtain the recursive schemes for the PMF $f_{q}(x;n,k;\theta)$ without using the $B_{q}(r,s,t,k)$.

\begin{theorem}
\label{recurpmftype4q-bino}
The PMF $f_{q}(x;n,k;\theta)=P_{q,\theta}(E_{n,k}=x)$, $x\in \mathcal{R}(E_{n,k}=x)$, $0<q\leq 1,$ satisfies for $n\geq k$ the recursive relation
\begin{equation}
\label{recursive1}
\begin{split}
f_{q}(x;n,k;\theta)=&\sum_{\substack{i=1\\i\neq k+1}}^{n}\theta^{i-1}(1-\theta)f_{q}(x;n-i,k;\theta q)\\
&+\theta^{k}(1-\theta)f_{q}(x-1;n-k-1,k;\theta q)\\
&+\theta^{k}\delta_{x,1}\delta_{n,k}+\theta^{n}\delta_{x,0}I(n>k),
\end{split}
\end{equation}
with $f_{q}(x;n,k;\theta)=0$ if $x<0$ or $x>\lfloor\frac{n+1}{k+1}\rfloor$ and $f_{q}(x;n,k;\theta)=\delta_{x,0}$ if $0\leq n <k$.
\end{theorem}

\begin{proof}
Obviously for $x<0$ or $x >\lfloor\frac{n+1}{k+1}\rfloor$ and for $0 \leq n < k$ the theorem holds. Let $Y$ be a geometric RV with parameter $1-\theta$, PMF $h_{Y}(i)$ and CDF $H_{Y}(i)$, $i = 1,2,\ldots.$ For $n\geq k$ we observe that
\begin{equation*}
\begin{split}
f_{q}(x;n,k;\theta)=&\sum_{i=1}^{n}P_{q,\theta}(E_{n,k}=x|Y=i)h_{Y}(i)\\
&+P_{q,\theta}(E_{n,k}=x|Y>n)(1-H_{Y}(n))\\
=&\sum_{i=1}^{k}\theta^{i-1}(1-\theta)f_{q}(x;n-i,k;\theta q)\\
&+\theta^{k}(1-\theta)f_{q}(x-1;n-k-1,k;\theta q)\\
&+\sum_{i=k+2}^{n}\theta^{i-1}(1-\theta)f_{q}(x;n-i,k;\theta q)\\
&+\theta^{k}\delta_{x,1}\delta_{n,k}+\theta^{n}\delta_{x,0}I(n>k)
\end{split}
\end{equation*}

\end{proof}

The alternative recursive scheme for $f_{q}(x;n,k;\theta)$, applying mathematical algebra, is obtained by following corollary.

\begin{corollary}
\label{recurpmftype4q-bino2}
The PMF $f_{q}(x;n,k;\theta)=P_{q,\theta}(E_{n,k}=x)$, $x\in \mathcal{R}(E_{n,k}=x)$, $0<q\leq 1,$ satisfies for $n> k+1$ the recursive relation
\begin{equation}
\label{recyrsivepmf2}
\begin{split}
f_{q}(x;n&,k;\theta)=\\
&\theta f_{q}(x;n-1,k;\theta)+(1-\theta)\Big[f_{q}(x;n-1,k;\theta q)+\\
&\theta^{k}\left\{f_{q}(x-1;n-k-1,k;\theta q)-f_{q}(x;n-k-1,k;\theta q)\right\}+\\
&\theta^{k+1}\left\{f_{q}(x;n-k-2,k;\theta q)-f_{q}(x-1;n-k-2,k;\theta q)\right\}\Big].\\
\end{split}
\end{equation}
with $f_{q}(x;n,k;\theta)=0$ if $x<0$ or $x>\lfloor\frac{n+1}{k+1}\rfloor$, $f_{q}(x;n,k;\theta)=\delta_{x,0}$ if $0\leq n <k$, $f_{q}(0;k,k;\theta)=1-\theta^{k},$ $f_{q}(1;k,k;\theta)=\theta^{k},$ $f_{q}(0;k+1,k;\theta)=1-\theta^{k}(1-\theta)(1+q^k)$ and $f_{q}(1;k+1,k;\theta)=\theta^{k}(1-\theta)(1+q^k).$

\end{corollary}

\subsection[PGF, MGF and moments of $E_{n,k}$]{PGF,MGF and moments of $E_{n,k}$}
In this section we derive the PDF, MGF and moments of $E_{n,k}$ using recursive scheme. We also derive the alternitive expression for the PMF of $E_{n,k}$.
\begin{proposition}
For $0< q\leq 1,$ the PGF $\phi_{q}(t;n,k;\theta)=E_{q,\theta}(t^{E_{n,k}})=\sum_{x\in \mathcal{R}(E_{n,k})}\\t^{x}f_{q}(x;n,k;\theta),$ $t\in R,$ for $n>k,$ is given by
\begin{equation}
\begin{split}
\phi_{q}(t;n,k;\theta)=\theta \phi_{q}(t;n-1,k;\theta)+&(1-\theta)\Big\{\phi_{q}(t;n-1,k;\theta q)\\
&+\theta^{k}(t-1)\phi_{q}(t;n-k-1,k;\theta q)\\
&+\theta^{k+1}(1-t)\phi_{q}(t;n-k-2,k;\theta q)\Big\}.\\
\end{split}
\end{equation}
with $\phi_{q}(x;n,k;\theta)=1$ for $0\leq n <k,$  $\phi_{q}(t;k,k;\theta)=1+\theta^{k}(t-1)$ and $\phi_{q}(t;k+1,k;\theta)=1-\theta^{k}(1-\theta)(1+q^k)+\theta^{k}(1-\theta)(1+q^k)t.$
\end{proposition}
\begin{proof}
Multiplying both sides of \eqref{recyrsivepmf2} by $t^{x}$ and summing up for all $x$ we obtain the recursive scheme for the PGF of $E_{n,k}$ under model.
\end{proof}

\begin{proposition}
For $0< q\leq 1,$ the MGF $\psi_{q}(t;n,k;\theta)=E_{q,\theta}(e^{tE_{n,k}})=
\sum_{x\in \mathcal{R}(E_{n,k})}e^{tx}\\ f_{q}(x;n,k;\theta)=\phi_{q}(e^{t};n,k;\theta),$ $t\in R,$ for $n>k,$ is given by
\begin{equation}
\begin{split}
\psi_{q}(t;n,k;\theta)=&\theta \psi_{q}(t;n-1,k;\theta)\\
&+(1-\theta)\Big\{\psi_{q}(t;n-1,k;\theta q)\\
&+\theta^{k}(e^{t}-1)\psi_{q}(t;n-k-1,k;\theta q)\\
&+\theta^{k+1}(1-e^{t})\psi_{q}(t;n-k-2,k;\theta q)\Big\}.\\
\end{split}
\end{equation}
with $\psi_{q}(x;n,k;\theta)=1$ for $0\leq n <k,$ $\psi_{q}(t;k,k;\theta)=1+\theta^{k}(e^{t}-1)$ and $\psi_{q}(t;k+1,k;\theta)=1-\theta^{k}(1-\theta)(1+q^k)+\theta^{k}(1-\theta)(1+q^k)e^{t}.$
\end{proposition}

\begin{corollary}
\label{fmm}
For $0< q\leq 1,$ let $\rho_{q}(r;n,k;\theta)=E_{q,\theta}\left[(E_{n,k})_{r}\right]=E_{q,\theta}\Big[E_{n,k}(E_{n,k}-1)(E_{n,k}-2)\cdots (E_{n,k}-r+1)\Big]$ and $\nu_{q}(r;n,k;\theta)=E_{q,\theta}(E_{n,k}^{r}),$ $r \geq 1.$ Then for $n>k,$
\begin{equation}
\begin{split}
\rho_{q}(r;n,k;\theta)=&\theta \rho_{q}(r,n-1,k;\theta)+(1-\theta)\Big\{\rho_{q}(r,n-1,k;\theta q)\\
&+\theta^{k}r\rho_{q}(r-1,n-k-1,k;\theta q)-\theta^{k+1}r\rho_{q}(r-1,n-k-2,k;\theta q)\Big\}.\\
\end{split}
\end{equation}
with $\rho_{q}(0;n,k;\theta)=1$ for $n \geq 0$; $\rho_{q}(r,n,k;\theta)=0$ for $0\leq n <k$; $\rho_{q}(r,n,k;\theta)=\theta^{k}\delta_{r,1}$ for $n=k,$; $\rho_{q}(r,n,k;\theta)=\theta^{k}(1-\theta)(1+q^k)\delta_{r,1}$ for $n=k+1,$ and
\noindent
\begin{equation}
\begin{split}
\nu_{q}(r;n,k;\theta)=&\theta \nu_{q}(r,n-1,k;\theta)+(1-\theta)\Bigg\{\nu_{q}(r,n-1,k;\theta q)\\
&+\theta^{k}\sum_{i=0}^{r-1}{r \choose i}\nu_{q}(i,n-k-1,k;\theta q)\\
&-\theta^{k+1}\sum_{i=0}^{r-1}{r \choose i}\nu_{q}(i,n-k-2,k;\theta q)\Bigg\}.\\
\end{split}
\end{equation}
\noindent
with $\nu_{q}(0;n,k;\theta)=1$ for $n \geq 0$; $\nu_{q}(r,n,k;\theta)=0$ for $0\leq n <k$; $\nu_{q}(r,n,k;\theta)=\theta^{k}$ for $n=k,$
$\nu_{q}(r,n,k;\theta)=\theta^{k}(1-\theta)(1+q^k)$ for $n=k+1.$

\end{corollary}

\begin{remark}
{\rm
Using Corollary \ref{fmm} an alternative expression for the PMF of $E_{n,k}$ is given by
\begin{equation}
f_{q}(x;n,k;\theta)=\frac{1}{x!}\sum_{r\geq x}(-1)^{r-x}\frac{\rho_{q}(r;n,k;\theta)}{(r-x)!}=\sum_{r\geq x}(-1)^{r-x}{r \choose x}\rho_{q}^{\prime}(r;n,k;\theta),
\end{equation}
where $\rho_{q}^{\prime}(r;n,k;\theta)=E_{q,\theta}\left[{E_{n,k} \choose r}\right]=\frac{\rho_{q}(r;n,k;\theta)}{r!}$ is the $r$th binomial moment of $E_{n,k}$. And an alternative expression for the survival(reliability) function of $E_{n,k}$ is given by
\begin{equation}
\sum_{i\geq x}f_{q}(i;n,k;\theta)=\sum_{i\geq x}(-1)^{x+i}{i-1 \choose x-1}\frac{\rho_{q}(i;n,k;\theta)}{i!}.
\end{equation}
}
\end{remark}
Moreover, the $r$th moment about the mean, $\xi_{q}(r;n,k;\theta)= E_{q,\theta}\left\{[E_{n,k}-E_{q,\theta}(E_{n,k})]^{r}\right\}$, is provided by the relation,
\begin{equation}
\begin{split}
\xi_{q}(r;n,k;\theta)=&\sum_{j=0}^{r-2}(-1)^{j}{r \choose j} \nu_{q}(r-j;n,k;\theta)(\nu_{q}(1;n,k;\theta))^{j}\\
&+(-1)^{r-1}(r-1)(\nu_{q}(1;n,k;\theta))^{r},\ r \geq 2,
\end{split}
\end{equation}
with $\xi_{q}(0;n,k;\theta)=1,$ $\xi_{q}(1;n,k;\theta)=0.$ Consequently, we get the shape factors $\gamma_{1}$ (a measure of skewness or asymmetry) and $\gamma_{2}$(a measure of kurtosis or peakedness)
\begin{equation}
\begin{split}
\gamma_{1}(n,k;q,\theta)&=\xi_{q}(3;n,k;\theta)/ \left(\xi_{q}(2;n,k;\theta)\right)^{3/2},\\
\gamma_{2}(n,k;q,\theta)&=\xi_{q}(4;n,k;\theta)/ \left(\xi_{q}(2;n,k;\theta)\right)^{2}
\end{split}
\end{equation}
\noindent
of a $q-B_{k}^{\rom{4}}(n,\theta)$ distribution.

\subsection{Closed formulae for the mean and variance of $E_{n,k}$}
In this section we obtain the closed formulae for the mean and variance of $E_{n,k}$. Alternative recurrent relations for the means and the variance of $E_{n,k}$ under model are obtained as $E_{q,\theta}(E_{n,k})=\rho_{q}(1;n,k;\theta)=\nu_{q}(1;n,k;\theta)$, $Var_{q,\theta}(E_{n,k})=\xi_{q}(2;n,k;\theta)=\nu_{q}(2;n,k;\theta)-(\nu_{q}(1;n,k;\theta))^{2}=\rho_{q}(2;n,k;\theta)+\rho_{q}(1;n,k;\theta)(1-\rho_{q}(l;n,k;\theta))$.

\begin{lemma}
\label{indicatotft}
For the indicators $I_{j}=(1-X_{j-1})(1-X_{j+k})\prod_{m=j}^{j+k-1}X_{m},$ $j=1,2,\ldots,n-k+1,$ defined on a $q$-sequence of
$0-1$ RVs $\{X_{j}\}_{j=1}^{n}$ obeying model, let $\mu_{j}=E_{q,\theta}(I_{j})$ and $\mu_{i,j}=E_{q,\theta}(I_{i}I_{j}).$ Then \noindent
\begin{equation*}
\begin{split}
\mu_{j}=&\theta^{k}\sum_{i=2}^{j}q^{(i-1)k}\left(1-\theta q^{i-2}\right)\left(1-\theta
q^{i-1}\right)b_{q}(j-i;j-2;\theta),\ \text{for}\ j=1,2,\ldots,n-k+1,\\
\end{split}
\end{equation*}
and for $i=1,2,\ldots,n-2k-1,$ $j=i+1,i+2,\ldots,n-k$, $\mu_{i,j}=0$ if $j-i\leq k$
and
\begin{equation*}
\mu_{i,j}=\begin{cases}
             \theta^{2k}\sum\limits_{\alpha=2}^{i}\sum\limits_{\beta=\alpha+2}^{\alpha+j-i-k}q^{(\alpha+\beta-2)k}\left(1-\theta q^{\alpha-2}\right)\left(1-\theta
q^{\alpha-1}\right)\\
b_{q}(i-\alpha;i-2;\theta)
\left(1-\theta q^{\beta-2}\right)\left(1-\theta
q^{\beta-1}\right)^{I(j<i+k+1)}\\
b_{q}(j-i-k-\beta+\alpha;j-2-i-k;\theta q^{\alpha}),\\
\multicolumn{1}{@{}r@{\quad}}{ \mbox{for $i\geq 2$ and $j>i+k+1$}} \\
\theta^{2k}\sum\limits_{\beta=3}^{j-k}q^{(\beta-1)k}(1-\theta)\left(1-\theta q^{\beta-2}\right)\left(1-\theta
q^{\beta-1}\right)^{I(j<i+k+1)}\\
b_{q}(j-\beta-k;j-2-i-k;\theta q),\\
\multicolumn{1}{@{}r@{\quad}}{ \hbox{for $i=1$ and $j>i+k+1$}} \\
\theta^{2k}\sum\limits_{\alpha=2}^{i}q^{(2\alpha-1)k}\left(1-\theta q^{\alpha-2}\right)\left(1-\theta
q^{\alpha-1}\right)b_{q}(i-\alpha;i-2;\theta),\\
\multicolumn{1}{@{}r@{\quad}}{ \hbox{for $i\geq 2$ and $j=n-k+1$}} \\
\theta^{2k}q^{k}\left(1-\theta \right)\left(1-\theta
q\right),\\
\multicolumn{1}{@{}r@{\quad}}{ \hbox{for $i=1$ and $j=n-k+1$}} \\
\end{cases}
\end{equation*}

%
%
%

\end{lemma}

\begin{proof}
 Appling the total probability law, for $j=2,\ldots,n-k,$ we have

\begin{equation*}
\begin{split} \mu_{j}=&P_{q,\theta}(I_{j}=1)=\sum_{i=2}^{j}P_{q,\theta}\left(I_{j}=1,\
S_{j+k}^{(1)}=j+k-i\right)\\ =&\sum_{i=2}^{j}P_{q,\theta}\left(X_{j-1}=0,\ X_{j}=\cdots=X_{j+k-1}=1,\ X_{j+k}=0,\ S_{j-2}^{(1)}=j-i\right)\\
=&\sum_{i=2}^{j}\left(1-\theta q^{i-2}\right)\left(\theta q^{i-1}\right)^{k}\left(1-\theta
q^{i-1}\right)P_{q,\theta}\left(S_{j-2}^{(1)}=j-i\right)\\ =&\sum_{i=2}^{j}\left(1-\theta q^{i-2}\right)\left(\theta
q^{i-1}\right)^{k}\left(1-\theta q^{i-1}\right)\left[\begin{array}{c}
         j-2\\
        j-i \\
        \end{array} \right]_{q}\theta^{j-i}{\prod_{{m}=1}^{i-2}}(1-\theta q^{m-1}).\\
\end{split}
\end{equation*}

In a similar way we get \begin{equation*}
\begin{split}
\mu_{1}=&P_{q,\theta}(I_{1}=1)=P_{q,\theta}\left(I_{1}=1,\
S_{k+1}^{(1)}=k\right)\\ =&P_{q,\theta}\left(X_{1}=\cdots=X_{k}=1,\ X_{k+1}=0\right)=\theta^{k}\left(1-\theta \right),\\ \end{split}
\end{equation*}
and
\begin{equation*}
\begin{split}
\mu_{n-k+1}=&P_{q,\theta}(I_{n-k+1}=1)=\sum_{i=1}^{n-k}P_{q,\theta}\left(I_{n-k+1}=1,\ S_{n}^{(1)}=n-i\right)\\
=&\sum_{i=1}^{n-k}P_{q,\theta}\left(X_{n-k}=0,\ X_{n-k+1}=\cdots=X_{n}=1,\ S_{n-k+1}^{(1)}=n-i-k\right)\\ =&\sum_{i=1}^{n-k}\left(1-\theta
q^{i-1}\right)\left(\theta q^{i}\right)^{k}\left(1-\theta q^{i-1}\right)P_{q,\theta}\left(S_{n-k+1}^{(1)}=n-i-k\right)\\
=&\theta^{k}\sum_{i=1}^{n-k}\left(1-\theta q^{i-1}\right)q^{ik}\left[\begin{array}{c}
         n-k-1\\
        n-k-i \\
        \end{array} \right]_{q}\theta^{n-k-i}{\prod_{{m}=1}^{i-1}}\left(1-\theta q^{m-1}\right).\\
\end{split}
\end{equation*}

Next, we observe that for $j-i\leq k,$ $\mu_{i,j}=0$ and for $i=2,\ldots,n-2k-1,$ $j=i+k+1,\ldots,n-k,$ we have

\begin{equation*}
\begin{split} \mu_{i,j}=&\sum_{\alpha=2}^{i}\sum_{\beta=\alpha+2}^{\alpha+j-i-k}P_{q,\theta}\left(I_{i}=1,\
I_{j}=1,\ S_{i+k}^{(1)}=i+k-\alpha,\ S_{j+k}^{(1)}=j+k-\beta\right)\\
=&\sum_{\alpha=2}^{i}\sum_{\beta=\alpha+2}^{\alpha+j-i-k}P_{q,\theta}\Big(X_{i-1}=0,\ X_{i}=\cdots=X_{i+k-1}=1,\ X_{i+k}=0,\
S_{i-2}^{(1)}=i-\alpha,\\ &X_{j-1}=0,\ X_{j}=\cdots=X_{j+k-1}=1,\ X_{j+k}=0,\ S_{j-2}^{(1)}=j-\beta\Big)\\
=&\sum_{\alpha=2}^{i}\sum_{\beta=\alpha+2}^{\alpha+j-i-k}\left(1-\theta q^{\alpha-2}\right)\left(\theta q^{\alpha-1}\right)^{k}\left(1-\theta q^{\alpha-1}\right)b_{q}(i-\alpha;i-2;\theta)\big(1-\theta q^{\beta-2}\big)\\
&\big(\theta q^{\beta-1}\big)^{k}\big(1-\theta q^{\beta-1}\big)b_{q}(j-i-k-\beta;j-2-i-k;\theta q^{\alpha}).
\end{split}
\end{equation*}
The other cases are derived in a similar way.

\end{proof}

\begin{remark}
{\rm
In the particular case, for $q\rightarrow 1$, $\begin{bmatrix}
                                                 n \\
                                                 x
                                               \end{bmatrix}_{q}$
converges to $\binom{n}{x}$, then we have $\mu_{1}\rightarrow \theta^k(1-\theta),$ $\mu_{n-k+1}\rightarrow \theta^k(1-\theta),$ $\mu_{j}\rightarrow \theta^k(1-\theta)^2,$ for $j=2,\ldots,n-k,$ $\mu_{i,j}\rightarrow 0$ for $j-i\leq k$, $\mu_{1,k+2}=\mu_{n-2k,n-k+1}\rightarrow \theta^{2k}(1-\theta)^2$, $\mu_{i,i+k+1}\rightarrow \theta^{2k}(1-\theta)^3$ for $i=2,3,\ldots,n-2k-1,$ $\mu_{1,j}\rightarrow \theta^{2k}(1-\theta)^3$ for $j=k+3,\ldots,n-k,$ $\mu_{i,n-k+1}\rightarrow \theta^{2k}(1-\theta)^3$ for $i=2,\ldots,n-2k-1,$ and $\mu_{i,j}\rightarrow \theta^{2k}(1-\theta)^4$ for $j-i\geq k+2.$
}
\end{remark}

\begin{theorem}
Let $mu_{j}$ and $\mu_{i,j}$ be as in Lemma \ref{indicatotft}. Then, for $n\geq k$ the mean, $E_{q,\theta}(E_{n,k})$, and the variance, $Var_{q,\theta}(E_{n,k})$, are given by
  \begin{equation*}
    E_{q,\theta}(E_{n,k})=\sum_{j=1}^{n-k+1}\mu_{j},
  \end{equation*}
and

  \begin{equation*}
    \begin{split}
    Var_{q,\theta}(E_{n,k})=&\sum_{i=1}^{n-k+1}\Bigg\{\mu_{i}\Bigg(1-\sum_{j=1}^{n-k+1}I_{j}\Bigg)\\
    &+I(n-2k-1)\Bigg(\sum_{j=1}^{i-k+1}\mu_{j,i}+\sum_{j=i+k+1}^{n-k+1}\mu_{i,j}\Bigg)\Bigg\}.
    \end{split}
  \end{equation*}

\end{theorem}

\begin{proof}
We start with the study of $E_{q,\theta}(E_{n,k})$. Since the linearity of expectations we have
  \begin{equation*}
    E_{q,\theta}(E_{n,k})=E_{q,\theta}\Bigg(\sum_{j=1}^{n-k+1}I_{j}\Bigg)=\sum_{j=1}^{n-k+1} E_{q,\theta}(I_{j})=\sum_{j=1}^{n-k+1}\mu_{j}.
  \end{equation*}
We are mow going to condition on $I_{i}=1$ we have that for $i=1,2,\ldots,n-k+1,$
\begin{equation*}
\begin{split}
 E_{q,\theta}(E_{n,k}\mid I_{i}=1)=&E_{q,\theta}\Bigg(\sum_{j=1}^{n-k+1}I_{j}\mid I_{i}=1\Bigg)=\sum_{j=1}^{n-k+1} E_{q,\theta}(I_{j}\mid I_{i}=1)\\
 =&1+\sum_{j=1}^{i-k+1}\frac{\mu_{j,i}}{\mu_{i}}+\sum_{j=i+k+1}^{n-k+1}\frac{\mu_{i,j}}{\mu_{i}},\ n\geq 2k+1,
\end{split}
  \end{equation*}

and $E_{q,\theta}(E_{n,k}\mid I_{i}=1)=1$ for $k\leq n\leq 2k$. Since $E_{n,k}$ can be expressed as the sum of the indicators $I_i$, $E_{n,k}^2$ written as follows.

\begin{equation*}
  \begin{split}
E_{n,k}^2=&I_1( I_1 + I_2 + I_3 + \cdots+ I_{n-k+1} ) +\\
          &I_2( I_1 + I_2 + I_3 + \cdots+ I_{n-k+1} ) +\\
          &I_3( I_1 + I_2 + I_3 + \cdots+ I_{n-k+1} ) +\\
          &\cdots\\
          &I_{n-k+1}( I_1 + I_2 + I_3 + \cdots+ I_{n-k+1} ).
  \end{split}
\end{equation*}
The expected value of $I_1( I_1 + I_2 + I_3 + \cdots+ I_{n-k+1})$ conditional on $I_i = 1$ is the same as the expected value of $I_1 + I_2 + I_3 + \cdots+ I_{n-k+1}$ conditional on $I_i = 1$, which is the expected value of $E_{n,k}$ under the same condition. If we condition on $I_i = 0$, the expected value
is just equal to 0. Now we use that the expected value is the sum
$P_{q,\theta}(Ii = 1)E_{q,\theta}(E_{n,k} | I_i = 1) + P_{q,\theta}(I_i = 0)E_{q,\theta}(E_{n,k} | I_i = 0)$ and you add
up over all $i = 1,2,...,n-k+1$, we have

  \begin{equation*}
    \begin{split}
    E_{q,\theta}(E_{n,k}^{2})=&\sum_{i=1}^{n-k+1}P_{q,\theta}(I_i = 1)E_{q,\theta}(E_{n,k} | I_i = 1)\\
=&\sum_{i=1}^{n-k+1}\Bigg\{\mu_{i}+\sum_{j=1}^{i-k+1}\mu_{j,i}+\sum_{j=i+k+1}^{n-k+1}\mu_{i,j}\Bigg\},\ n\geq 2k+1,
    \end{split}
  \end{equation*}

and $E_{q,\theta}(E_{n,k}^{2})=\sum_{i=1}^{n-k+1}P_{q,\theta}(I_i = 1)=\sum_{i=1}^{n-k+1}\mu_{i}$, for $k\leq n \leq 2k$.

\end{proof}

\begin{remark}
{\rm
In the particular case, for $q\rightarrow 1$, $\begin{bmatrix}
                                                 n \\
                                                 x
                                               \end{bmatrix}_{q}$
converges to $\binom{n}{x}$, then we have $E_{q,\theta}(E_{n,k})\rightarrow E_{\theta}(E_{n,k})=(1-\theta)\theta^2 \{2+(n-k-1)(1-\theta)\},$ for $k=1,\ldots,n-1$ and $E_{q,\theta}(E_{n,k})\rightarrow E_{\theta}(E_{n,k})=\theta^n$, for $k=n$.
}
\end{remark}

\section{Simulation Study} \label{sec: simul}
The Maximum Likelihood Estimation (MLE) is one of the most commonly used estimation methods because it has good properties: invariance, consistency, asymptotic efficiency, and asymptotic normality (see \cite{CaseBerg:01} for the details). In a sequence of $n$-binary (0-1) trials with varying probability of ones, one can define a random variable and a parametric model on runs of ones. \cite{makri2016runs} illustrated the likelihood inference from Type \RNum{2} $q$-binomial distribution of order $k$ model. Likewise, we conduct a simulation study to explain the likelihood inference from Type \RNum{4} $q$-binomial distribution of order $k$ model.

We assume that the true value of $q$ is known and $\theta$ is parameter of interest. We generate $N$ independent $q$-binomial sequences of length $n$, when the true value of $\theta$ is $\theta_0$. For each sequence, compute $E_{n,k}^{(i)}; i=1, \dots, N$, the number of non-overlap runs of length $k$. Let $\boldsymbol{E}_N=(E_{n,k}^{(1)}, \dots, E_{n,k}^{(N)})^T$. Then, $l(\theta; \boldsymbol{E}_N)$, the log-likelihood function of $\theta$ given the simulated data $\boldsymbol{E}_N$, is as follows:
\begin{equation} \label{eq:log_like_fun}
    l(\theta;~ \boldsymbol{E}_N)=\sum_{i=1}^N \log(f(E_{n,k}^{(i)};~ \theta)), \text{ for }\theta \in [0,~ 1].
\end{equation}
$\hat{\theta}$, the MLE of $\theta$ is as follows:
\begin{equation} \label{eq:mle}
    \hat{\theta}=\argmax_{\theta \in [0,~1]} l(\theta;~ \boldsymbol{E}_N).
\end{equation}
To solve the optimization in \eqref{eq:mle}, we use the Brent's derivative-free algorithm \citep{brent2002algorithms}.

For the interval estimation, we employ the likelihood ratio (LR) confidence interval. $I_{1-\alpha}$, $(1-\alpha)100\%$ LR confidence interval is as follows:
\begin{equation} \label{eq:lr_int}
    I_{1-\alpha}=\{\theta \in [0,~1]:~2l(\hat{\theta};~ \boldsymbol{E}_N)-2l(\theta;~ \boldsymbol{E}_N) \le \chi^2_{1,\alpha}\},
\end{equation}
where $\chi^2_{1, \alpha}$ is $1-\alpha$ quantile of $\chi^2$ distribution of degree of freedom of one, i.e.,\ $P(X \le \chi^2_{1,\alpha})=1-\alpha$ for a random variable $X$ from  $\chi^2$ distribution of degree of freedom of one and $\alpha \in (0,~1)$. Under certain regularity conditions, the set $I_{1-\alpha}$ is a closed interval, i.e.,\ $I_{1-\alpha}=[L_\alpha, U_\alpha]$, where $L_\alpha$, and $U_\alpha$ are solutions to the equation
\begin{equation} \label{eq:LU_int}
    2l(\hat{\theta};~ \boldsymbol{E}_N)-2l(\theta;~ \boldsymbol{E}_N) =\chi^2_{1,\alpha}.
\end{equation}
We apply Monte Carlo (MC) simulation of size $M$, in which, the same data generation and statistical inference processes are repeated for $M$ times independently. For each of $M$ MC samples, we generate $\boldsymbol{E}_{N,m}$ and compute $\hat{\theta}_m$ and $I_{\alpha, m}$ for $m=1, \dots, M$. From those values, we compute the following measures.
\begin{enumerate}
    \item Bias:
    \begin{equation}\label{eq:bias}
            \bar{\hat{\theta}}-\theta_0,
    \end{equation}
    where $\bar{\hat{\theta}}=\frac{\sum_{m=1}^M\hat{\theta}_m}{M}$

    \item Standard Error (SE):
    \begin{equation}\label{eq:SE}
        \sqrt{{\sum_{m=1}^M \Big(\hat{\theta}_m-\bar{\hat{\theta}}\Big)^2}/{M}}.
    \end{equation}
    \item Root Mean Squared Errors (RMSE):
    \begin{equation}\label{eq:RMSE}
        \sqrt{{\sum_{m=1}^M (\hat{\theta}_m-\theta_0)^2}/{M}}.
    \end{equation}

    \item Coverage Probability (CP):
    \begin{equation}\label{eq:CP}
            \frac{\sum_{m=1}^M, I(\theta_0 \in I_{\alpha, m})}{M}.
    \end{equation}
    \item Mean width of the confidence interval of $\theta$ (MW):
    \begin{equation}\label{eq:width}
        \frac{\sum_{m=1}^M(\max(I(\theta_0 \in I_{\alpha, m}))-\min(I(\theta_0 \in I_{\alpha, m})))}{M}
    \end{equation}

\end{enumerate}

All possible combinations of the following true parameters are investigated in this simulation study: $q \in \{0.6, ~0.8\}$, $n \in \{11, ~15, ~20, ~25\}$, $\theta \in \{0.05,~0.1,~0.15, \dots,~0.95\}$, $k \in \{3,~5\}$, and $N \in \{100,~1000\}$.

\begin{figure}[htbp]
\centering
\subfloat[$q=0.6$, $k=3$, and $N=100$]{
  \includegraphics[width=65mm]{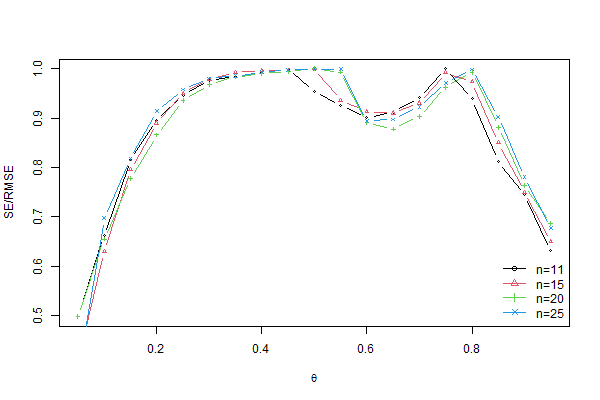}
}
\subfloat[$q=0.6$, $k=3$, and $N=1000$]{
  \includegraphics[width=65mm]{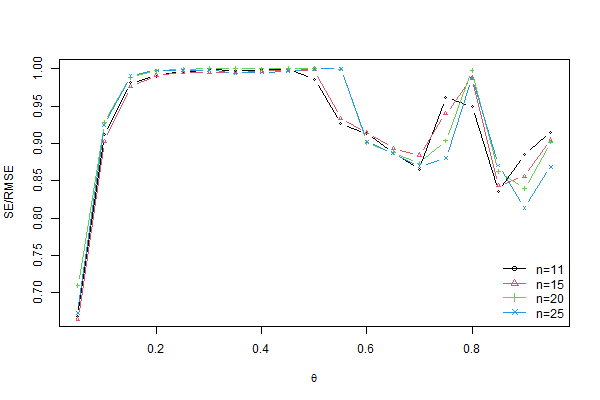}
}
\hspace{0mm}
\subfloat[$q=0.6$, $k=5$, and $N=100$]{
  \includegraphics[width=65mm]{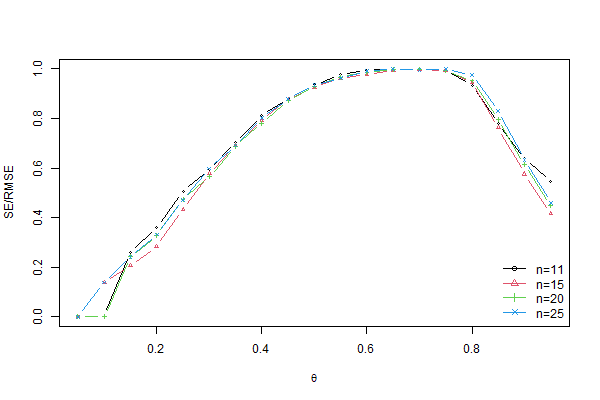}
}
\subfloat[$q=0.6$, $k=5$, and $N=1000$]{
  \includegraphics[width=65mm]{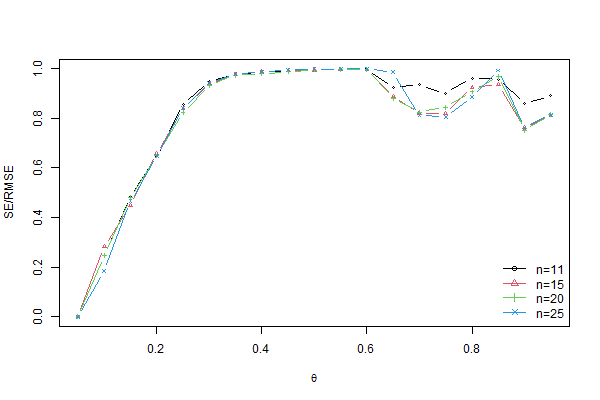}
}
\hspace{0mm}
\subfloat[$q=0.8$, $k=3$, and $N=100$]{   
  \includegraphics[width=65mm]{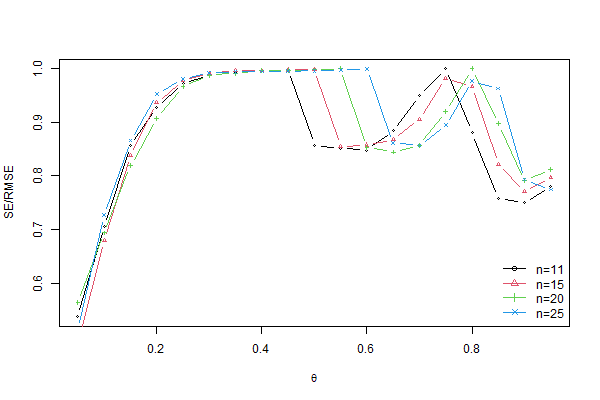}
}
\subfloat[$q=0.8$, $k=3$, and $N=1000$]{
  \includegraphics[width=65mm]{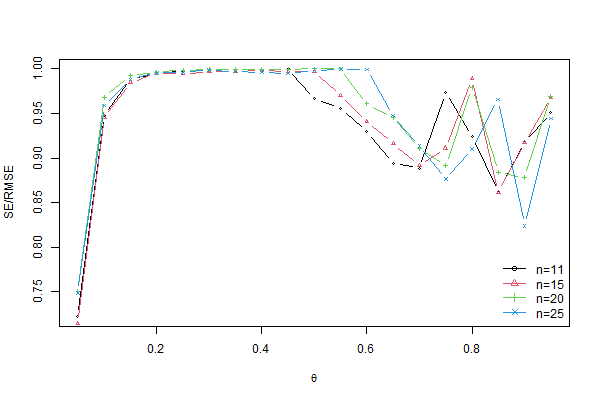}
}
\hspace{0mm}
\subfloat[$q=0.8$, $k=5$, and $N=100$]{   
  \includegraphics[width=65mm]{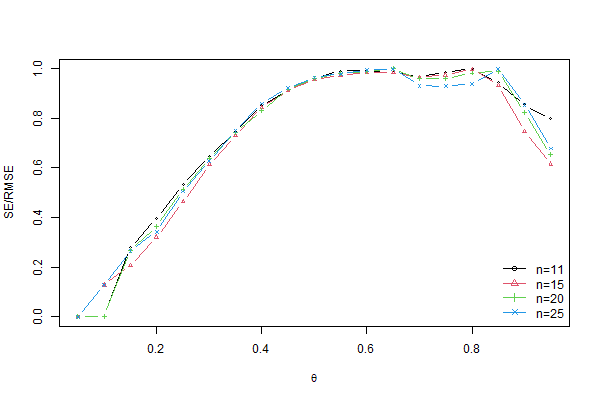}
}
\subfloat[$q=0.8$, $k=5$, and $N=1000$]{
  \includegraphics[width=65mm]{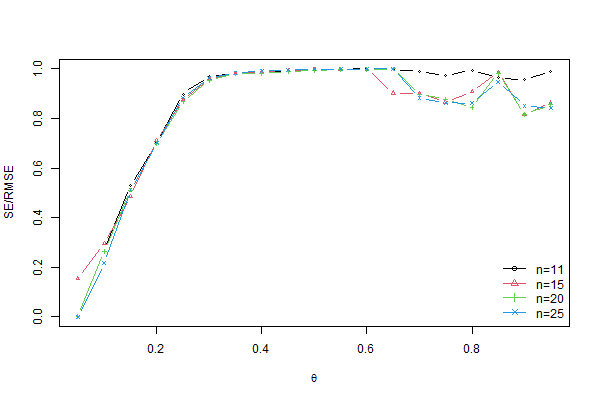}
}
\caption{The plots of $SE/RMSE$ vs $\theta$.} \label{fig:SE/RMSE}
\end{figure}

Figure \ref{fig:SE/RMSE} presents the plots of $SE/RMSE$ vs $\theta$. As illustrated in the graphs, when $\theta$ is close to 0 or 1, the $SE/RMSE$ value is far from 1, implying a greater bias in the MLE. However, when the sample size is raised from 100 to 1000, the $SE/RMSE$ value approaches unity for extreme values of $\theta$. Additionally, it is worth noting that there are cases where the value of $SE/RMSE$ is far from 1, despite the fact that $\theta$ is moderate.

\begin{figure}[htbp]
\centering
\subfloat[$q=0.6$, $k=3$, and $N=100$]{
  \includegraphics[width=65mm]{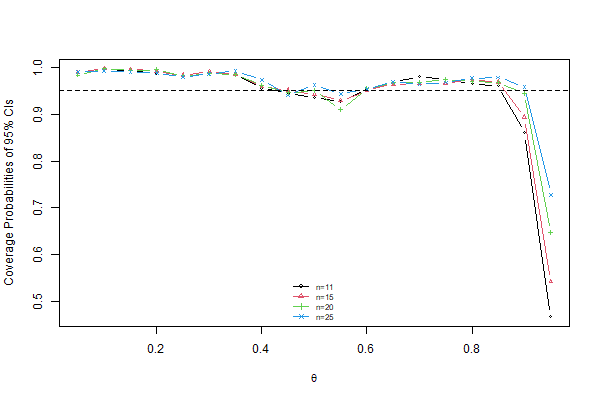}
}
\subfloat[$q=0.6$, $k=3$, and $N=1000$]{
  \includegraphics[width=65mm]{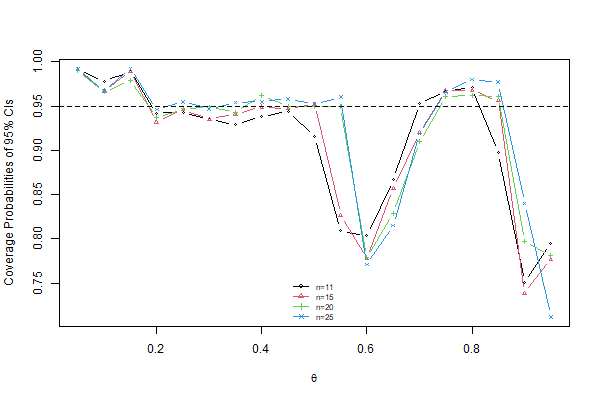}
}
\hspace{0mm}
\subfloat[$q=0.6$, $k=5$, and $N=100$]{
  \includegraphics[width=65mm]{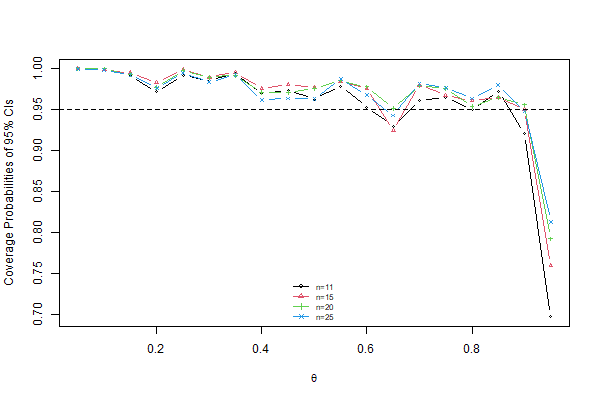}
}
\subfloat[$q=0.6$, $k=5$, and $N=1000$]{
  \includegraphics[width=65mm]{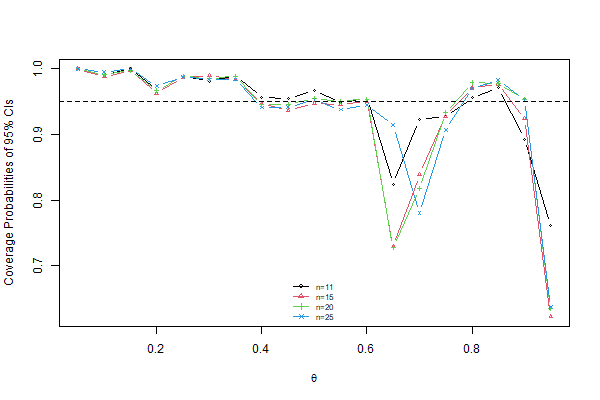}
}
\hspace{0mm}
\subfloat[$q=0.8$, $k=3$, and $N=100$]{   
  \includegraphics[width=65mm]{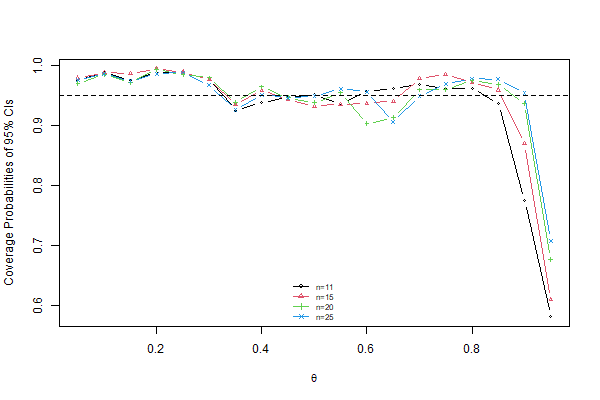}
}
\subfloat[$q=0.8$, $k=3$, and $N=1000$]{
  \includegraphics[width=65mm]{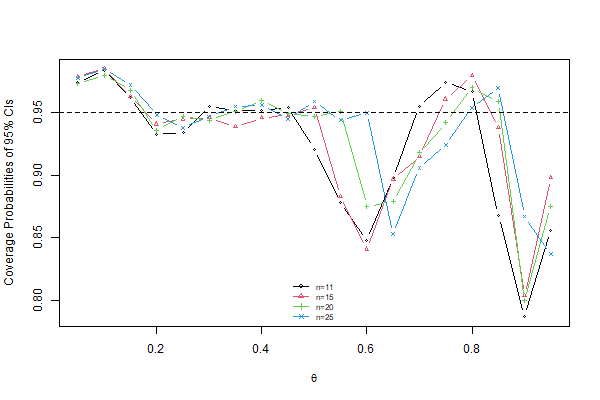}
}
\hspace{0mm}
\subfloat[$q=0.8$, $k=5$, and $N=100$]{   
  \includegraphics[width=65mm]{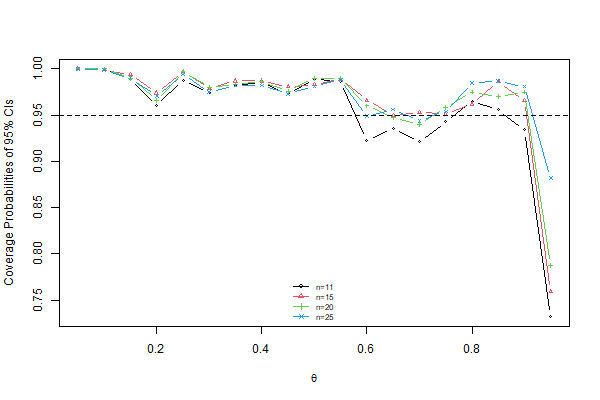}
}
\subfloat[$q=0.8$, $k=5$, and $N=1000$]{
  \includegraphics[width=65mm]{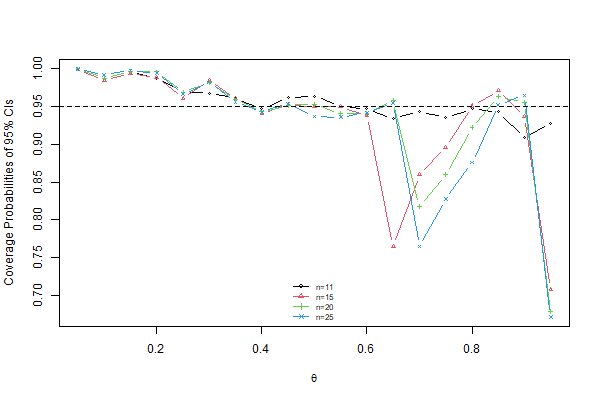}
}
\caption{The plots of the coverage probabilities of 95\% CIs vs $\theta$.} \label{fig:CP}
\end{figure}

The figures in Figure \ref{fig:CP} show the coverage probabilities of 95 percent confidence intervals versus $\theta$. Consistent with the results in Figure \ref{fig:SE/RMSE}, when $\theta$ is close to 0 or 1, or even moderate, the coverage probabilities depart from the nominal 0.95 coverage probability. Taken together, these findings imply that the MLE of $\theta$ can be inconsistent in the type $4$ $q$-binomial model.

\begin{figure}[htbp]
\centering
\subfloat{
  \includegraphics[width=65mm]{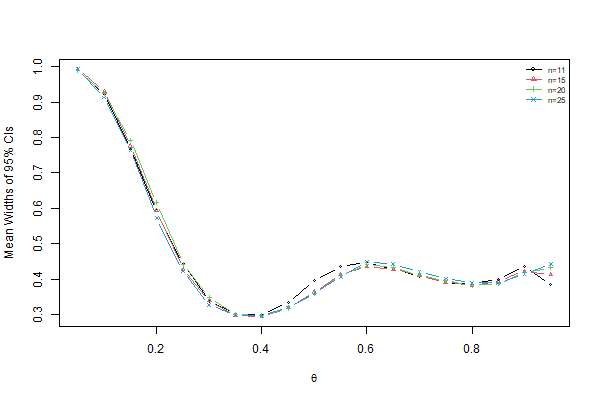}
}
\subfloat{
  \includegraphics[width=65mm]{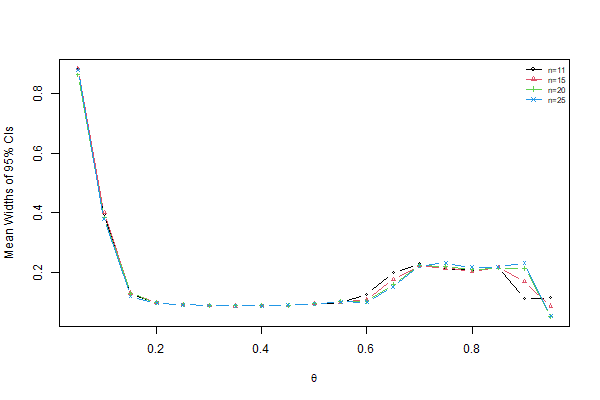}
}
\hspace{0mm}
\subfloat{
  \includegraphics[width=65mm]{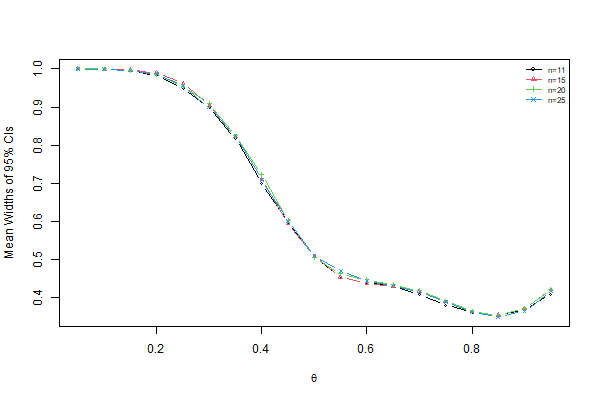}
}
\subfloat{
  \includegraphics[width=65mm]{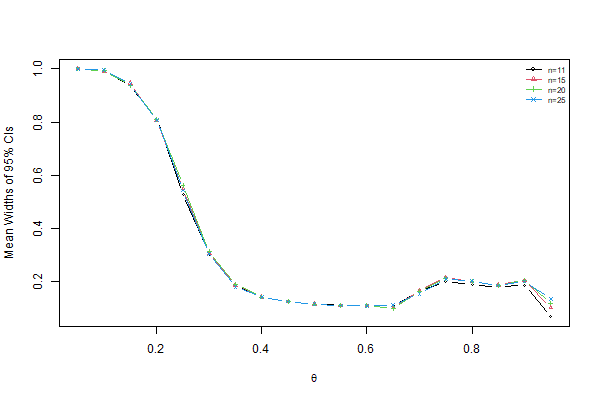}
}
\hspace{0mm}
\subfloat{   
  \includegraphics[width=65mm]{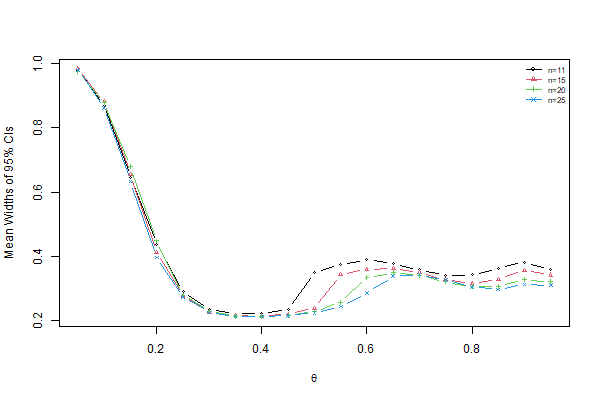}
}
\subfloat{
  \includegraphics[width=65mm]{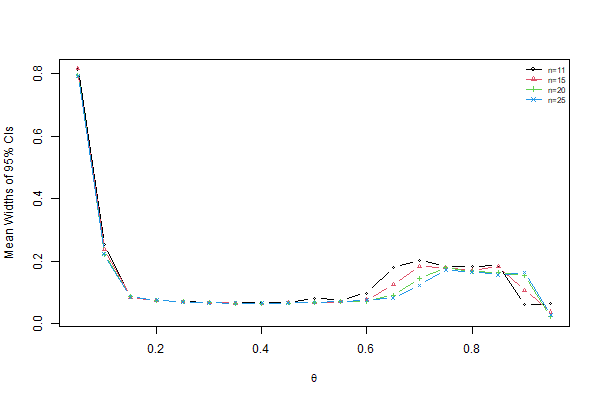}
}
\hspace{0mm}
\subfloat{   
  \includegraphics[width=65mm]{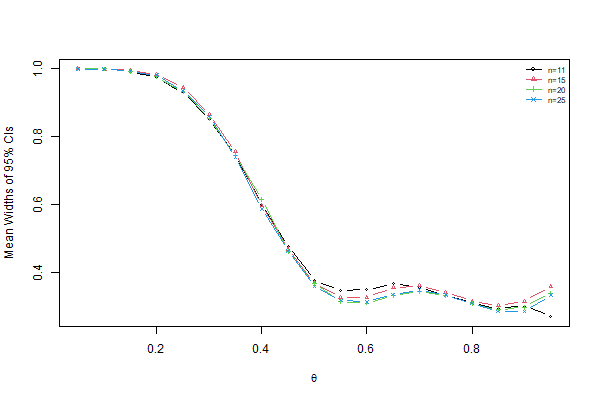}
}
\subfloat{
  \includegraphics[width=65mm]{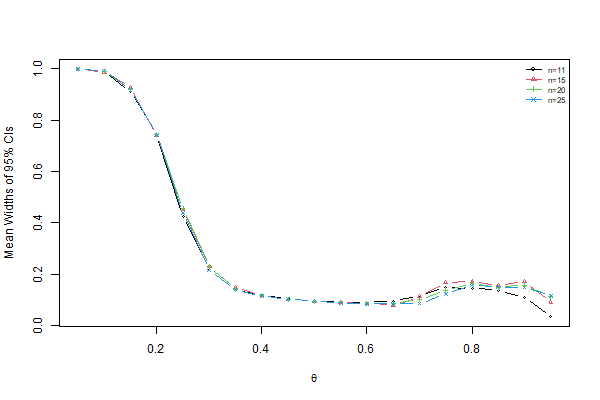}
}
\caption{The plots of mean widths of 95\% CIs vs $\theta$.} \label{fig:MW}
\end{figure}

The plots in Figure 3 depict the mean widths of 95 percent confidence intervals vs $\theta$. There is an evident trend that the mean width of the CI is wider for tiny values of $\theta$.

\bibliographystyle{agsm}
\bibliography{biblio-combine.bib}

\end{document}